\documentclass[oneside,english,reqno]{amsart}
\usepackage[T1]{fontenc}
\usepackage[latin9]{inputenc}
\usepackage{geometry}
\usepackage{amsthm}
\usepackage{amssymb}

\makeatletter

\newcommand*\LyXZeroWidthSpace{\hspace{0pt}}

\numberwithin{equation}{section}
\numberwithin{figure}{section}

\theoremstyle{definition}
\newtheorem{definition}{Definition}[section]

\theoremstyle{remark}
\newtheorem{rem}[definition]{Remark}

\theoremstyle{plain}
\newtheorem{thm}{Theorem}

\newtheorem{prop}[definition]{Proposition}
\newtheorem{lem}[definition]{Lemma}

\usepackage{amsthm}

\usepackage{dsfont}
\usepackage[hidelinks]{hyperref}

\usepackage[textsize=footnotesize,textwidth=0.85in]{todonotes}

\makeatletter
\newcommand{\vast}{\bBigg@{3}}
\newcommand{\Vast}{\bBigg@{4}}
\makeatother

\date{}

\makeatother

\usepackage{babel}

\begin{document}
\global\long\def\DR#1#2{{\rm DR}_{#1}\left(#2\right)}%
\global\long\def\DRR#1#2{{\rm DR}_{#1}^{1}\left(#2\right)}%

\title{Meromorphic differentials and twisted DR hierarchies for the Hodge CohFT }

\author{Xavier Blot}
\address{X.~B.: Korteweg-de Vriesinstituut voor Wiskunde, Universiteit van Amsterdam, Postbus 94248, 1090 GE Amsterdam, Nederland}
\email{x.j.c.v.blot@uva.nl}

\author{Paolo Rossi}
\address{P.~R.:Dipartimento di Matematica "Tullio Levi-Civita". 
Universit\`a degli Studi di Padova, 
Via Trieste 63,
35121 Padova, Italy.}
\email{paolo.rossi@math.unipd.it}

\author{Adrien Sauvaget}
\address{A.~S: Laboratoire AGM,
Universit\'e de Cergy Pontoise,
2 avenue Adolphe Chauvin,
95302 Cergy-Pontoise Cedex, France.}
\email{adrien.sauvaget@math.cnrs.fr}

\begin{abstract}
In \cite{blot2024meromorphic}, two families of classical and quantum integrable hierarchies associated to arbitrary Cohomological Field Theories (CohFTs) were introduced: the meromorphic differential and twisted double ramification hierarchies. For trivial CohFT, the authors established a connection with the untwisted Double Ramification (DR) hierarchy. In this paper, we extend this study to the Hodge CohFT and prove an analogous correspondence with the untwisted DR hierarchy. This yields non-trivial identities between Hodge integrals over the DR cycle, the twisted DR cycle and the cycle of meromorphic differentials.
\end{abstract}

\maketitle
\tableofcontents{}

\section{Introduction}

\subsection{Preliminaries}

\subsubsection{Strata of differentials, twisted and untwisted double ramification cycles.}

Let $g$ and $n$ be two nonnegative integers satisfying $2g-2+n>0$. We denote by $\mathcal{M}_{g,n}$ the moduli space of smooth curves of genus $g$ with $n$ marked points, and denote by $\overline{\mathcal{M}}_{g,n}$ the Deligne-Mumford compactification of $\mathcal{M}_{g,n}$. We index the $n$ marked points from $0$ to $n-1$. Let $m_{0},\dots,m_{n-1}$ be $n$ integers satisfying 
\[
\sum_{i=0}^{n-1}m_{i}=2g-2.
\]
Define $\mathcal{H}_{g}\left(m_{0},\dots,m_{n-1}\right)\subset\mathcal{M}_{g,n}$ as the locus of marked smooth curves $\left(C,x_{0},\dots,x_{n-1}\right)$ such that the canonical bundle satisfies $\omega_{C}\cong\mathcal{O}_{C}\left(\sum_{i=0}^{n-1}m_{i}x_{i}\right).$ We denote by
\[
\overline{\mathcal{H}}_{g}\left(m_{0},\dots,m_{n-1}\right)
\]
the closure of this locus in $\overline{\mathcal{M}}_{g,n}$, and refer to it as the\emph{ stratum of differentials}. If any $m_{i}$ is negative, this stratum is called \emph{meromorphic}, otherwise it is called \emph{holomorphic}. We use the same notation to refer to the associated homology class. Moreover, we denote by
\[
\DRR g{m_{0},\dots,m_{n-1}}
\]
the \emph{twisted double ramification cycle}.  defined as $2^{-g}{\rm P}_{g}^{g,1}\left(m_{0}+1,\dots,m_{n-1}+1\right)$, where ${\rm P}_{g}^{d,k}\in H^{2d}\left(\overline{\mathcal{M}}_{g,n}\right)$ is the Pixton class defined in \cite[Section 1.1]{JPPZ17}. By the results of~ \cite{Holmes,BHPSS2020pixton}, it may alternatively be defined via a resolution of the Abel--Jacobi map, or a cycle constructed the classes of strata and push-forwards along boundary components of $\overline{\mathcal{M}}_{g,n}$ associated with star-graphs. 

On the other hand, if $a_{0},\dots,a_{n-1}$ are integers satisfying
$
\sum_{i=0}^{n-1}a_{i}=0,
$
then we denote by
\[
\DR g{a_{0},\dots,a_{n-1}}
\]
 the (untwisted) \emph{double ramification cycle}. It is constructed either as the push forward of the virtual fundamental class of the moduli space of rubber maps to $\mathbb{P}^{1}$ \cite{JPPZ17}, or as in the twisted case by resolving the Abel-Jacobi map \cite{Holmes}, and it can be expressed by the Pixton class $2^{-g}{\rm P}_{g}^{g,0}\left(a_{0},\dots,a_{n-1}\right)$, see \cite{JPPZ17}. 

\medskip

\subsubsection{Previous work.}

In \cite{blot2024meromorphic}, the authors introduced two families of classical and quantum integrable hierarchies associated to arbitrary Cohomological Field Theories (CohFTs): the \emph{meromorphic differential hierarchy}, and the \emph{twisted double ramification hierarchy}. In both cases, the quantum Hamiltonian densities are defined via the same Hodge integrals, evaluated over strata of meromorphic differentials in the former, and over twisted DR cycles in the latter.

For the trivial CohFT, the two hierarchies are identified via a change of variables. Moreover, the meromorphic differential hierarchy coincides with the top degree part of the (untwisted) double ramification hierarchy \cite{Buryak15,BR2016}, yielding the well-known KdV hierarchy at the classical level, as well as a new quantization for it.  

\medskip

\subsubsection{Present work.}

In this paper, we extend the study of both hierarchies to the Hodge CohFT defined by
\[
\Lambda\left(\mu\right):=1+\mu\lambda_{1}+\cdots+\mu^{g}\lambda_{g},
\]
where $\lambda_{j}$ is the $j$th Chern class of the Hodge bundle and $\mu$ is a formal variable. 

We adopt all definitions and notations from \cite{blot2024meromorphic}, unless otherwise stated, and refer the reader to that work for the underlying constructions. Recall that for $i=0,\dots,n-1$, $\psi_{i}$ denotes the first Chern class of the line bundle over $\overline{\mathcal{M}}_{g,n}$ whose fiber at a marked curve is the cotangent line at its $i$-th marked point. Let $d\geq-1$, the quantum Hamiltonians densities of the hierarchy of meromorphic differentials for the Hodge CohFT are
\[
H_{d}(x)=\sum_{\substack{g,n\geq0\\
2g+n>0
}
}\!\frac{(i\hbar)^{g}}{n!} \! \sum_{m_{i}\in\mathbb{Z}} \!\left(\int_{\overline{\mathcal{H}}_{g}(-1,m_{1},\dots,m_{n},\,2g-1-\sum m_{i})}\!\psi_0^{d+1}\Lambda\!\left(\tfrac{-\epsilon^{2}}{i\hbar}\right)\Lambda(\mu)\right)q_{m_{1}}\cdots q_{m_{n}}(ix)^{\sum m_{i}-2g},
\]
while the quantum Hamiltonians densities of the hierarchy of twisted DR are defined by replacing the strata of meromorphic differentials with the corresponding twisted DR cycles: 
\[
H_{d}^{{\rm DR}^{1}}\!\!\left(x\right) = \!\!\sum_{\substack{g,n\geq0\\
2g+n>0
}
}\!\frac{(i\hbar)^{g}}{n!} \! \sum_{m_{i} \in\mathbb{Z}} \!\left(\int_{\DRR g{-1,m_{1},\dots,m_{n},\,2g-1-\sum m_{i}}}\!\psi_0^{d+1}\Lambda\!\left(\tfrac{-\epsilon^{2}}{i\hbar}\right)\Lambda(\mu)\right)q_{m_{1}}\cdots q_{m_{n}}(ix)^{\sum m_{i}-2g}.
\]
Recall that these quantum Hamiltonians are elements of
\[
\mathcal{B}[[\hbar]]=\mathbb{C}[q_{1},q_{2}\dots][[q_{0},q_{-1},\dots]][[ix,\left(ix\right)^{-1}]][[\epsilon,\hbar]].
\]
Let $\mathcal{A}^{{\rm sing}}[[\hbar]]=\mathbb{C}[[u_{0}]][u_{1},u_{2}.\dots;x^{-1}][[\epsilon,\hbar]]$ be the so-called ring of \emph{singular differential polynomials}. There is an injective map $\mathcal{A}^{{\rm sing}}[[\hbar]]\rightarrow\mathcal{B}[[\hbar]]$ defined on the generators by
\begin{equation}
u_{s}^{\alpha}\mapsto\sum_{m\in\mathbb{Z}}\underset{m^{\underline{s}}}{\underbrace{m\left(m-1\right)\cdots\left(m-s+1\right)}}q_{m}^{\alpha}i^{m}x^{m-s},\qquad\epsilon\mapsto\epsilon,\quad\hbar\mapsto\hbar.\label{eq: change variable q to u}
\end{equation}
The symbol $m^{\underline{s}}$ is called a \emph{falling factorial}. When the integrals are polynomials in $m_{1},\dots,m_{n}$, these Hamiltonian densities belong to the image of this map, and we interpret this map as a change of variables between $u$-variables and $q$-variables. We refer to \cite[Section 1]{blot2024meromorphic} for the definition of the star product, quantum bracket and the Poisson bracket in the classical setting.


\medskip

\subsection{Polynomiality of integrals over meromorphic strata\label{subsec: result polynomiality}}

\begin{prop}
\label{prop:polynomiality}Fix $g,n,d\geq0$ such that $2g-1+n>0$ and fix $0\leq l_{1},l_{2}\leq g$. The integral
\begin{equation}
\int_{\overline{\mathcal{H}}_{g}\left(2g-2-\sum m_{i},m_{1},\dots,m_{n}\right)}\psi_{0}^{d}\lambda_{l_{1}}\lambda_{l_{2}},
\end{equation}
such that $2g-2-\sum_{i=1}^{n}m_{i}<0$, is a polynomial in the variables $m_{1},\dots,m_{n}$ of degree $2g$. 
\end{prop}


As a consequence the quantum Hamiltonians of the hierarchy of meromorphic differentials can be rewritten as singular differential polynomials in terms of the $u$-variables:
\begin{align}
H_{d} & =\sum_{\substack{g,n\geq0\\
2g+n>0
}
}\frac{\left(i\hbar\right)^{g}}{n!}\sum_{\substack{s_{1},\dots,s_{n}\geq0\\
s_{1}+\cdots+s_{n}\leq2g
}
}\frac{\left(-1\right)^{g}}{x^{2g-\sum s_{i}}}\label{eq: Hd in u var}\\
 & \times\left[m_{1}^{\underline{s_{1}}}\cdots m_{n}^{\underline{s_{n}}}\right]\left(\int_{\overline{\mathcal{H}}_{g}\left(-1,m_{1},\dots,m_{n},2g-1-\sum m_{i}\right)}\psi_{0}^{d+1}\Lambda\left(\frac{-\epsilon^{2}}{i\hbar}\right)\Lambda\left(\mu\right)\right)u_{s_{1}}\cdots u_{s_{n}}.\nonumber 
\end{align}

\begin{rem}
\label{rem:CohFT-poly}
It will be apparent from the proof that this polynomiality property  holds when one Hodge class is replaced by any CohFT $(V,e_1,\eta,c_{g,n})$, where $V$ is a vector space, $e_1 \in V$ a unit, $\eta$ is a nondegenerate bilinear form and $c_{g,n}:V^{\otimes n}\rightarrow H^{*}\left(\overline{\mathcal{M}}_{g,n}\right)$, satisfying :
$$\int_{\overline{\mathcal{H}}_{g}\left(m_{1},\dots,m_{n}\right)}\lambda_{l}c_{g,n}\left(v_1,\dots,v_n\right)=0,\quad{\rm for}\,n\neq1, \text{ $m_1,\dots,m_n \geq 0,$ and $v_i \in V$. } $$
\end{rem}

\medskip

\subsection{Main result\label{subsec:Main-result}} We first formulate the main theorem of the present work is in terms of intersection numbers along strata of differentials and double ramification cycles. 
\begin{thm}
[main theorem]\label{thm: main theorem}Fix $g,n\geq0$ such that $2g+n>0$, fix $d\geq-1$ and fix $0\leq l_{1},l_{2}\leq g$. Let $\left(s_{1},\dots,s_{n}\right)$ be a list of nonnegative integers. The coefficient of $m_{1}^{\underline{s_{1}}}\cdots m_{n}^{\underline{s_{n}}}$ in the the polynomial
\begin{equation}
\int_{\overline{\mathcal{H}}_{g}\left(-1,m_{1},\dots,m_{n},2g-2-\sum m_{i}+1\right)}\psi_{0}^{d+1}\lambda_{l_{1}}\lambda_{l_{2}}\label{eq: Hodge integrals differentials}
\end{equation}
vanishes if $\sum_{i=1}^{2g}s_{i}\neq2g$, and when $\sum_{i=1}^{2g}s_{i}=2g$ it is given by the coefficient of $a_{1}^{s_{1}}\cdots a_{n}^{s_{n}}$ in the polynomial
\begin{equation}
\int_{\DR g{0,a_{1},\dots,a_{n},-\sum a_{i}}}\psi_{0}^{d+1}\lambda_{l_{1}}\lambda_{l_{2}}.
\end{equation}
In particular, $\int_{\overline{\mathcal{H}}_{g}\left(-1,m_{1},\dots,m_{n},2g-2-\sum m_{i}+1\right)}\psi_{0}^{d+1}\lambda_{l_{1}}\lambda_{l_{2}}$ is a homogeneous factorial polynomial of degree $2g$.
\end{thm}

\begin{rem}
Recall the elementary fact that a polynomial is uniquely written as a factorial polynomial (i.e. in terms of falling factorials) of the same degree, and conversely. In the statement above, the integrals (\ref{eq: Hodge integrals differentials}) are polynomial in the variable $m_{i}$ by Proposition~\ref{prop:polynomiality}. We can therefore rewrite this polynomial in the falling factorial basis and extract $m_{1}^{\underline{s_{1}}}\cdots m_{n}^{\underline{s_{n}}}$. 
\end{rem}

\begin{rem}
We expect the theorem to hold when one Hodge class is replaced by any CohFT satisfying the vanishing conditions stated in Remark~\ref{rem:CohFT-poly}.
\end{rem}

Theorem~\ref{thm: main theorem} is proved in Section~\ref{sec: proof main} using the following equivalent reformulation in terms of integrable hierarchy.

\subsection{Equivalent formulation in terms of integrable hierarchy\label{subsec:Equivalent-formulation-in}}

\subsubsection{Quantum hierarchies.}

The main theorem can be reformulated as
\begin{equation}
H_{d}=\left(H_{d}^{{\rm DR}}\right)^{\left[0\right]},\quad d\geq-1,\label{eq: main thm in terms of hamiltonians}
\end{equation}
where $H_{d}^{{\rm DR}}$ denotes the quantum Hamiltonian of the (non-twisted) DR hierarchy associated to the Hodge CohFT given by 
\begin{align}
H_{d}^{{\rm DR}} & =\sum_{\substack{g,n\geq0\\
2g+n>0
}
}\frac{\left(i\hbar\right)^{g}}{n!}\sum_{\substack{s_{1},\dots,s_{n}\geq0\\
s_{1}+\cdots+s_{n}\leq2g
}
}\left(-i\right)^{s_{1}+\cdots+s_{n}}\\
 & \times\left[a_{1}^{s_{1}}\cdots a_{n}^{s_{n}}\right]\left(\int_{\DR g{0,a_{1},\dots,a_{n},-\sum a_{i}}}\psi_{0}^{d+1}\Lambda\left(\frac{-\epsilon^{2}}{i\hbar}\right)\Lambda\left(\mu\right)\right)u_{s_{1}}\cdots u_{s_{n}}.\nonumber 
\end{align}
Here, the notation $[0]$ means extracting the degree $0$ in $H_{d}^{{\rm DR}}$, where the degree of a singular differential polynomial is determined by ${\rm deg}\,u_{i}^{\alpha}=i$, ${\rm deg}\,\epsilon=-1$, ${\rm deg}\,\frac{1}{x}=1$, and ${\rm deg}\,\hbar=-2$. 

Thus, all apparent negative powers of $x$ vanish in the expression of $H_{d}$ in Eq.~(\ref{eq: Hd in u var}) and, therefore, $H_{d}$ is an actual differential polynomial, that is $H_{d}\in\mathbb{C}[[u_{0}]][u_{>0}][[\epsilon,\hbar]]$.

In addition, since the highest degree of $H_{d}^{{\rm {DR}}}$ is zero and since the quantum bracket of the meromorphic differential hierarchy is the highest degree part of the quantum bracket of the DR hierarchy (see \cite[Remark 2.7]{blot2024meromorphic}), we conclude that, for the Hodge CohFT, the hierarchy of meromorphic differentials $(\left(H_{d}\right)_{d\geq0},\left[\cdot,\cdot\right])$ is the highest degree part of the DR hierarchy $(\left(H_{d}\right)_{d\geq0}^{{\rm DR}},\left[\cdot,\cdot\right]^{{\rm DR}})$.

\smallskip

\subsubsection{Classical hierarchies.}

In the classical limit $\left(\hbar=0\right)$, the classical Hamiltonian densities of the DR hierarchy $H_{d}^{{\rm DR}}\vert_{\hbar=0}$ are homogeneous of degree $0$. Therefore they equal the classical Hamiltonian densities of the hierarchy of meromorphic differential $H_{d}\vert_{\hbar=0}$. In addition, the Poisson bracket of the two hierarchies coincide, so the two classical hierarchies are identified.

In the case of the Hodge CohFT, Buryak computed the DR hierarchy and established that this hierarchy is equivalent to the Dubrovin-Zhang hierarchy which controls simple Hodge integrals \cite{Buryak15}. This is an instance of the DR/DZ equivalence proved in full generality in \cite{blot2024strong,blot2024master}. He also proved in \cite{Buryak_Hodge15} that the first equation of the deformed KdV hierarchy is the Intermediate Long Wave (ILW ) equation (up to a rescaling), and gave a precise relation with an infinite sequence of symmetries of the ILW equation discovered in \cite{satsuma1979internal}.

\medskip

\subsection{Relation between hierarchies of meromorphic differentials and twisted DR\label{subsec:Relation-between-hierarchies}}

The hierarchies of meromorphic differentials and twisted DR have the same quantum bracket (and therefore the same Poisson bracket in the classical setting). For the Hodge CohFT their quantum Hamiltonian densities are related by the following formula. 
\begin{prop}
\label{prop:Link-DR1-MD}Let $d\geq-1$. We have 
\[
H_{d}^{{\rm DR}^{1}}=H_{d}\bigg\vert_{u\left(x\right):=u\left(x\right)+\sum_{g\geq 1}\frac{c_{g}\left(\epsilon,\hbar,\mu\right)}{x^{2g}}},
\]
where
\[
c_{g}\left(\epsilon,\hbar,\mu\right)=\left(2g-1\right)\left(\epsilon^{2g}\mu^{g-1}-i\hbar\epsilon^{2g-2}\mu^{g}\right)\int_{\overline{\mathcal{H}}_{g}\left(2g-2\right)}\lambda_{g}\lambda_{g-1}.
\]
\end{prop}

This statement is proved in Section~\ref{sec:Proof polynomiality and link DR1}. 

\subsection{A closed formula for Hodge integrals over meromorphic strata\label{subsec:A-closed-formula}}

The proof of the main theorem is reduced to the following nice identity. 
\begin{prop}
\label{prop: Nice identity}Let $g\geq1$. We have
\begin{equation}
\int_{\overline{\mathcal{H}}_{g}\left(-1,m,2g-1+m\right)}\psi_{0}\lambda_{g}\lambda_{g-1}=m^{\underline{2g}}\frac{\left|B_{2g}\right|}{\left(2g\right)!},
\end{equation}
where $B_{2g}$ is the $2g$-Bernoulli number.
\end{prop}

This statement is proved in Section~\ref{sec:nice-identity}. 

\medskip

\subsection{Acknowledgment}

The development of this project, in particular the main theorem, owes much to computer experiments performed using the admcycles package \cite{delecroix2022admcycles} to compute various intersection numbers with DR cycles and strata of differentials.

X.B. was supported by the Dutch Research Council grant OCENW.M.21.233.

P. R. is supported by the University of Padova and is affiliated to the INdAM group GNSAGA.

\subsection{Data availability and conflict of interest statement}

Data sharing is not applicable to this article as no datasets were generated or analyzed during the current study.

The authors declare that they have no conflict of interest.

\medskip

\section{Proof of Proposition~\ref{prop:polynomiality} and Proposition~\ref{prop:Link-DR1-MD} \label{sec:Proof polynomiality and link DR1}}

We show how Proposition~\ref{prop:polynomiality} and Proposition~\ref{prop:Link-DR1-MD} follow from the following lemma. 
\begin{lem}
\label{lem:DR1-Mero-via-Conj A}Fix $g,n,l\geq0,$ such that $2g+n>0,$ and $0\leq l\leq g$, and fix $d\geq0$. Let $m_{1},\dots,m_{n}\in\mathbb{Z}$ be integers satisfying the inequality $2g-2-\sum_{i=1}^{n}m_{i}<0$. Let $\mu$ and $\epsilon$ be formal variables, and define $\Lambda\left(\mu\right):=1+\mu\lambda_{1}+\cdots+\mu^{g}\lambda_{g}$, and similarly for $\Lambda\left(\epsilon\right)$. The following identity holds:
\begin{align}
\int_{{\rm DR}_{g}^{1}\left(2g-2-\sum_{i=1}^{n}m_{i},m_{1},\dots,m_{n}\right)}\psi_{0}^{d}\Lambda\left(\mu\right) & \Lambda\left(\epsilon\right)=\sum_{s\geq0}\sum_{\substack{g_{0}\geq0\\
g_{1},\dots,g_{s}>0\\
g_{0}+g_{1}+\cdots+g_{s}=g
}
}\frac{1}{s!}\label{eq:link-DR1-Diff-Conj-A}\\
 & \times\int_{\overline{\mathcal{H}}_{g_{0}}\left(2g-2-\sum_{i=1}^{n}m_{i},m_{1},\dots,m_{n}-2g_{1},\dots,-2g_{s}\right)}\psi_{0}^{d}\Lambda\left(\mu\right)\Lambda\left(\epsilon\right)\nonumber \\
 & \times\prod_{i=1}^{s}\left(2g_{i}-1\right)\int_{\overline{\mathcal{H}}_{g_{i}}\left(2g_{i}-2\right)}\lambda_{g}\lambda_{g-1}\left(\epsilon^{g_{i}}\mu^{g_{i}-1}+\mu^{g_{i}}\epsilon^{g_{i}-1}\right).\nonumber 
\end{align}
\end{lem}

\begin{proof}
This is a direct generalization of \cite[Lemma 7.1]{blot2024meromorphic} obtained by replacing the trivial CohFT with the Hodge CohFT. We briefly recall the argument. The key input is the so-called Conjecture A in the appendix of \cite{farkas2018moduli}, which has been proved in \cite{BHPSS2020pixton}. In \cite[Appendix A.4]{farkas2018moduli}, a class
\[
{\rm H}_{g}\left(m_{0},\dots,m_{n}\right)
\]
is defined and we refer to this source for its definition and notations. The statement of Conjecture A is
\[
\DRR g{m_{0},\dots,m_{n}}={\rm H}_{g}\left(m_{0},\dots,m_{n}\right)
\]
if there exists a negative integer in the list $m_{0},\dots,m_{n}$. Intersecting this equality, for $m_{0}\LyXZeroWidthSpace=2g-2-\sum_{i=1}^{n}m_{i}$, with $\psi_{0}^{d}\Lambda\left(\mu\right)\Lambda\left(\epsilon\right)$ and integrating yields 
\begin{equation}
\int_{{\rm DR}_{g}^{1}\left(2g-2-\sum_{i=1}^{n}m_{i},m_{1},\dots,m_{n}\right)}\psi_{0}^{d}\Lambda\left(\mu\right)\Lambda\left(\epsilon\right)=\int_{{\rm H}_{g}\left(2g-2-\sum_{i=1}^{n}m_{i},m_{1},\dots,m_{n}\right)}\psi_{0}^{d}\Lambda\left(\mu\right)\Lambda\left(\epsilon\right).\label{eq:Conj A evaluated}
\end{equation}
The class ${\rm H}_{g}\left(2g-2-\sum_{i=1}^{n}m_{i},m_{1},\dots,m_{n}\right)$ is a sum over weighted star graphs such that the central vertex is decorated with a meromorphic stratum and each satellite vertex is decorated with a holomorphic stratum. Therefore the pole with weight $2g-2-\sum_{i=1}^{n}m_{i}$ belongs to the central vertex, and in particular the class $\psi_0^{d}$ decorates this central vertex. Moreover, the full Hodge class, being a CohFT, splits naturally when intersected with a boundary stratum. A dimension count shows that the right hand side of Eq.~(\ref{eq:Conj A evaluated}) is nonzero only if each satellite vertex  is attached to the central vertex by a unique edge, and can only contribute as $\int_{\overline{\mathcal{H}}_{g_{i}}\left(2g_{i}-2\right)}\lambda_{g}\lambda_{g-1}$, where $g_{i}>0$ is the genus of the satellite vertex. In particular, the twist of this edge is $2g_{i}-1$. Assembling all contributions of ${\rm H}_{g}$ yields the desired formula.
\end{proof}
\begin{proof}
[Proof of Proposition~\ref{prop:polynomiality}] The polynomiality property of 
\[
\int_{\overline{\mathcal{H}}_{g}\left(2g-2-\sum m_{i},m_{1},\dots,m_{n}\right)}\psi_{0}^{d}\lambda_{l_{1}}\lambda_{l_{2}}
\]
follows from Lemma~\ref{lem:DR1-Mero-via-Conj A} via an induction argument based on the fact that $$\int_{{\rm DR}_{g}^{1}\left(2g-2-\sum_{i=1}^{n}m_{i},m_{1},\dots,m_{n}\right)}\psi_{0}^{d}\Lambda\left(\mu\right)\Lambda\left(\epsilon\right)$$ is a polynomial of degree $2g$ in $m_{1},\dots,m_{n}$ \cite{spelier2024polynomiality,PixtonZagier}. 
\end{proof}
\begin{proof}
[Proof of Proposition~\ref{prop:Link-DR1-MD}]We substitute $m_{n}=2g-1-\sum_{i=0}^{n-1}m_{i}$ and $\mu=\frac{-\epsilon^{2}}{i\hbar}$ in Eq.~(\ref{eq:link-DR1-Diff-Conj-A}) to obtain 
\begin{align*}
H_{d}^{{\rm DR}^{1}} 
 & =\exp\left(\sum_{g>0}\left(2g-1\right)\int_{\overline{\mathcal{H}}_{g}\left(2g-2\right)}\lambda_{g}\lambda_{g-1}\left(\left(-\epsilon^{2}\right)^{g}\mu^{g-1}+\mu^{g}\left(-\epsilon^{2}\right)^{g-1}i\hbar\right)\frac{\partial}{\partial q_{-2g}}\right)H_{d}.
\end{align*}
Moreover the operator $\frac{\partial}{\partial q_{-2g}}$ acting on singular differential polynomial is written in $u$-variables as $\frac{\partial}{\partial q_{-2g}}=\sum_{k\geq0}\left(\partial_{x}\right)^{k}\left(\frac{1}{ix}\right)^{2g}\frac{\partial}{\partial u_{k}}$. Therefore, using the power series identity $\exp\left(t\partial_{x}\right)f\left(x\right)=f\left(x+t\right)$ we get 
\[
H_{d}^{{\rm DR}^{1}}=H_{d}\Bigg\vert_{\substack{u_{0}\rightarrow u_{0}+\sum_{g>1}\frac{c_{g}\left(\epsilon,\hbar,\mu\right)}{x^{2g}}\\
\quad u_{1}\rightarrow u_{1}+\partial_{x}\sum_{g>1}\frac{c_{g}\left(\epsilon,\hbar,\mu\right)}{x^{2g}}\\
\quad\vdots
}
}=H_{d}\Bigg\vert_{u\left(x\right)\rightarrow u\left(x\right)+\sum_{g \geq 1}\frac{c_{g}\left(\epsilon,\hbar,\mu\right)}{x^{2g}}},
\]
where we used 
$
c_{g}\left(\epsilon,\hbar,\mu\right)=\left(2g-1\right)\left(\epsilon^{2g}\mu^{g-1}-i\hbar\mu^{g}\epsilon^{2g-2}\right)\int_{\overline{\mathcal{H}}_{g}\left(2g-2\right)}\lambda_{g}\lambda_{g-1}.
$
\end{proof}

\section{Proof of the main theorem\label{sec: proof main}}

The main theorem is equivalent to the equality
\begin{equation}
H_{d}=\left(H_{d}^{{\rm DR}}\right)^{\left[0\right]},\qquad d\geq-1,
\end{equation}
that we establish now. 

We briefly recall part of the notation introduced in \cite[Section 2]{blot2024meromorphic}, to which we refer for the complete framework. Given a function $f\in\mathcal{B}[[\hbar]]$, the associated local functional is defined by $\int f\,dx:=\mathrm{Coef}_{(ix)^{-1}}\left(f-f|_{q_{*}=0}\right),$ and we write $\overline{f}:=\int f\,dx$ as shorthand. The variational derivative of a local functional is given by $\frac{\delta\overline{f}}{\delta u}:=\sum_{m\in\mathbb{Z}}(ix)^{-m-1}\frac{\partial\overline{f}}{\partial q_{m}}.$ When $f$ is an element of $\mathcal{A}=\mathbb{C}[[u_{0}]][u_{1},u_{2}.\dots][[\epsilon,\hbar]]$, the local functional $\overline{f}$ is identified with the image of $f$ via the quotient map $\mathcal{A}\rightarrow\ensuremath{\mathcal{A}/({\rm Im}\,(\partial_{x}:\mathcal{A}\to\mathcal{A})\oplus\mathbb{C}[[\epsilon,\hbar]])}$, where $\partial_{x}=\sum_{i\geq0}u_{i+1}\partial/\partial u_{i}$. In addition, the variational derivative is written as $\frac{\delta\overline{f}}{\delta u}=\sum_{s\geq0}\left(-\partial_{x}\right)^{s}\frac{\partial f}{\partial u_{s}}$ in $u$-variables.

Let $d\geq1$, we introduce $G_{d}$ and $G_{d}^{{\rm DR}}$ by 
\begin{equation}
G_{d}=\sum_{\substack{g,n\geq0\\
2g-1+n>0
}
}\frac{\left(i\hbar\right)^{g}}{n!}\sum_{m_{1},\dots,m_{n}\in\mathbb{Z}}P_{g,d}\left(m_{1},\dots,m_{n}\right)q_{m_{1}}\cdots q_{m_{n}}\left(ix\right)^{\sum_{i=1}^{n}m_{i}-2g},
\end{equation}
where
\begin{equation}
P_{g,d}\left(m_{1},\dots,m_{n}\right)=\int_{\overline{\mathcal{H}}_{g}\left(2g-2-\sum m_{i},m_{1},\dots,m_{n}\right)}\psi_{0}^{d}\Lambda\left(\frac{-\epsilon^{2}}{i\hbar}\right)\Lambda\left(\mu\right),
\end{equation}
and
\begin{equation}
G_{d}^{{\rm DR}}=\sum_{\substack{g,n\geq0\\
2g-1+n>0
}
}\frac{\left(i\hbar\right)^{g}}{n!}\sum_{s_{1},\dots,s_{n}\geq0}\left[a_{1}^{s_{1}}\cdots a_{n}^{s_{n}}\right]\left(\int_{\DR g{-\sum_{i=1}^{n}a_{i},a_{1},\dots,a_{n}}}\psi_{0}^{d}\Lambda\left(\frac{-\epsilon^{2}}{i\hbar}\right)\Lambda\left(\mu\right)\right)u_{s_{1}}\cdots u_{s_{n}}.
\end{equation}
They satisfy
\begin{equation}
\overline{G}_{d}=\overline{H}_{d},\quad\frac{\delta\overline{G}_{d}}{\delta u}=H_{d-1},
\end{equation}
and similarly
\begin{equation}
\overline{G_{d}^{{\rm DR}}}=\overline{H_{d}^{{\rm DR}}},\quad\frac{\delta\overline{G_{d}^{{\rm DR}}}}{\delta u}=H_{d-1}^{{\rm DR}}.
\end{equation}

We prove the identity
\begin{equation}
\overline{G}_{d}=\left(\overline{G_{d}^{{\rm DR}}}\right)^{\left[0\right]},\qquad d\geq0,\label{eq:equalitu of G gamilto}
\end{equation}
which implies the desired equality after applying the variational derivative. Using \cite[Lemma 8.1]{blot2024meromorphic}, Eq.~(\ref{eq:equalitu of G gamilto}) reduces to verifying the following three statements:
\begin{enumerate}
\item Matching genus-zero terms:
\begin{equation}
{\rm Coef}_{\epsilon^{0}\hbar^{0}}\overline{G}_{d}={\rm Coef}_{\epsilon^{0}\hbar^{0}}\left(\overline{G_{d}^{DR}}\right)^{\left[0\right]}=\frac{u^{d+2}}{\left(d+2\right)!},\quad d\geq0,
\end{equation}
where ${\rm Coef}_{\epsilon^{0}\hbar^{0}}$ stands for extracting the coefficient of $\epsilon^{0}\hbar^{0}$. 
\item Commutativity of Hamiltonians:
\begin{equation}
\left[\overline{G_{d}},\overline{G_{1}}\right]=0=\left[\left(\overline{G_{d}^{DR}}\right)^{\left[0\right]},\left(\overline{G_{1}^{DR}}\right)^{\left[0\right]}\right],\quad d\geq0,
\end{equation}
where the bracket corresponds to the commutator of the star product introduced in \cite[Section 2]{blot2024meromorphic}. 
\item Matching of the first Hamiltonian:
\begin{equation}
\overline{G}_{1}=\left(\overline{G_{1}^{{\rm DR}}}\right)^{\left[0\right]}=\int\left(\frac{u^{3}}{3!}+O\left(\epsilon,\hbar,\mu\right)\right).
\end{equation}
\end{enumerate}

\subsubsection*{Item $1$}

follows from a direct computation, recalling that in genus $0$ the codimension of the strata of meromorphic differentials and of the DR cycle is $0$.

\subsubsection*{Item $2$}

follows from the quantum integrability of the hierarchy of meromorphic differentials \cite{blot2024meromorphic} and of the DR hierarchy \cite{BR2016}. 

\subsubsection*{Item $3$}

requires more explanations. First, observe that in both cases the integrals involved are dimensionally nontrivial only when
\[
3=\left(g-l\right)+\left(g-k\right)+n,
\]
where $0\leq l,k\leq g$ are the degrees of the two Hodge classes $\lambda_{l}$ and $\lambda_{k}$. The equality 
\begin{equation}
\int_{\DR g{0,-a,a}}\psi_{0}\lambda_{g}\lambda_{g-1}=a^{2g}\frac{\left|B_{2g}\right|}{\left(2g\right)!}
\end{equation}
can be found in \cite[Section 4.3.2]{Buryak15}. This yields the explicit formula: 
\begin{align}
\overline{G_{1}^{{\rm DR}}} & =\int\left(\frac{u_{0}^{3}}{3!}+\sum_{g\geq1}\epsilon^{2g}\mu^{g-1}\frac{\left|B_{2g}\right|}{2\left(2g\right)!}u_{2g}u_{0}-\sum_{g\geq1}i\hbar\epsilon^{2\left(g-1\right)}\mu^{g}\frac{\left|B_{2g}\right|}{2\left(2g\right)!}u_{2g}u_{0}\right).
\end{align}
On the other hand, using Proposition~\ref{prop: Nice identity}, proven in the next section, we have 
\begin{equation}
\int_{\overline{\mathcal{H}}_{g}\left(-1,m,2g-1+m\right)}\psi_{0}\lambda_{g}\lambda_{g-1}=m^{\underline{2g}}\frac{\left|B_{2g}\right|}{\left(2g\right)!}.
\end{equation}
Therefore we recover the same expression
\begin{equation}
\overline{G}_{1}=\int\left(\frac{u_{0}^{3}}{3!}+\sum_{g\geq1}\epsilon^{2g}\mu^{g-1}\frac{\left|B_{2g}\right|}{2\left(2g\right)!}u_{2g}u_{0}-\sum_{g\geq1}i\hbar\epsilon^{2\left(g-1\right)}\mu^{g}\frac{\left|B_{2g}\right|}{2\left(2g\right)!}u_{2g}u_{0}\right),
\end{equation}
which concludes the proof.
\medskip

\section{Proof of the closed formula for Hodge integrals over meromorphic strata\label{sec:nice-identity}}

We prove the identity of Proposition~\ref{prop: Nice identity}: whenever $g\geq1$, we have
\[
\int_{\overline{\mathcal{H}}_{g}\left(-1,m,2g-1+m\right)}\psi_{0}\lambda_{g}\lambda_{g-1}=m^{\underline{2g}}\frac{\left|B_{2g}\right|}{\left(2g\right)!}.
\]

We have already established that $\int_{\overline{\mathcal{H}}_{g}\left(-1,m,2g-1-m\right)}\psi_{0}\lambda_{g}\lambda_{g-1}$ is a polynomial of degree $2g$. Moreover, this polynomial has $2g$ obvious zeros:
\begin{equation}
\int_{\overline{\mathcal{H}}_{g}\left(-1,m,2g-1-m\right)}\psi_{0}\lambda_{g}\lambda_{g-1}=0\quad{\rm for}\quad0\leq m\leq2g-1,
\end{equation}
since each case has a unique pole which is simple, therefore the stratum of differentials is empty by the residue condition. By interpolation we deduce that 
\begin{equation}
\int_{\overline{\mathcal{H}}_{g}\left(-1,m,2g-1-m\right)}\psi_{0}\lambda_{g}\lambda_{g-1}=C\times m^{\underline{2g}},
\end{equation}
for some constant $C$. Thus, proving Proposition~\ref{prop: Nice identity} amounts to prove the following formula 
\begin{equation}
I_{g}:=\int_{\overline{\mathcal{H}}_{g}\left(-1,2g,-1\right)}\psi_{0}\lambda_{g}\lambda_{g-1}=\left|B_{2g}\right|,
\end{equation}
which we shall now demonstrate.

\subsection{Step 1: Application of Conjecture A}

We first compute $\int_{\DRR g{-1,2g,-1}}\psi_{0}\lambda_{g}\lambda_{g-1}$using the expression of the twisted DR cycle as a sum over star graphs decorated with cycles of differentials conjectured in \cite[Conjecture A]{farkas2018moduli}, and proved in \cite{BHPSS2020pixton}. We refer to these sources for its formulation. The class $\lambda_{g}\lambda_{g-1}$ vanishes on the complement of the moduli space of rational tails curves (parametrizing curves with one irreducible component of geometric genus $g$). Therefore, the only non vanishing contributions to the integral are the following.
\begin{itemize}
\item The trivial graph and its contribution is
\begin{equation}
\int_{\overline{\mathcal{H}}_{g}\left(-1,2g,-1\right)}\psi_{0}\lambda_{g}\lambda_{g-1}=I_{g}.
\end{equation}
\item The graph with central vertex of genus $0$ with three legs, connected by an edge to a genus $g$ satellite vertex with no marking. The genus $0$ vertex is decorated with the cycle of meromorphic differentials $\overline{\mathcal{H}}_{0}\left(-1,-1,2g,-2g\right)=\overline{\mathcal{M}}_{0,4}$, and the genus $g$ by the cycle of holomorphic differentials $\overline{\mathcal{H}}_{g}\left(2g-2\right)$. Its contribution is
\begin{equation}
\left(2g-1\right)\cdot\int_{\overline{\mathcal{M}}_{0,4}}\psi\cdot\int_{\overline{\mathcal{H}}_{g}\left(2g-2\right)}\lambda_{g}\lambda_{g-1}.
\end{equation}
\end{itemize}
Therefore, we obtain
\begin{equation}
I_{g}=\int_{{\rm DR}_{g}^{1}\left(-1,2g,-1\right)}\psi_{0}\lambda_{g}\lambda_{g-1}-\left(2g-1\right)\int_{\overline{\mathcal{H}}_{g}\left(2g-2\right)}\lambda_{g}\lambda_{g-1}.
\end{equation}

We use once again Conjecture A, this time to compute $\int_{\DRR g{2g-1,-1}}\lambda_{g}\lambda_{g-1}$. The contribution of the trivial graph is $\int_{\overline{\mathcal{H}}_{g}\left(2g-1,-1\right)}\lambda_{g}\lambda_{g-1}=0$, since the space of meromorphic differentials, having a single simple pole, is empty due to the residue condition. The only non vanishing contribution is given by the graph with a central vertex of genus $0$ and two legs connected by an edge to a genus $g$ satellite vertex with no leg. We find
\begin{equation}
\int_{\DRR g{2g-1,-1}}\lambda_{g}\lambda_{g-1}=\left(2g-1\right)\int_{\overline{\mathcal{H}}_{g}\left(2g-2\right)}\lambda_{g}\lambda_{g-1},
\end{equation}
so that
\begin{equation}
I_{g}=\int_{{\rm DR}_{g}^{1}\left(-1,2g,-1\right)}\psi_{0}\lambda_{g}\lambda_{g-1}-\int_{\DRR g{2g-1,-1}}\lambda_{g}\lambda_{g-1}.
\end{equation}

We have used so far the convention that the weights $m_{0},\dots,m_{n-1}$ of the cycle $\DRR g{m_{0},\dots,m_{n-1}}$ satisfy $\sum_{i=0}^{n-1}m_{i}=2g-2$. 

\medskip

\subsection{Step 2: Splitting formula for the twisted DR cycle}

In \cite[Proposition 3.1]{CSS}, Costantini, Sauvaget and Schmitt establish a splitting formula for the twisted DR cycle when intersected with a $\psi$-class. We apply this formula to get 
\begin{equation}
\int_{{\rm DR}_{g}^{1}\left(a-1,2g-a,-1\right)}\psi_{0}\lambda_{g}\lambda_{g-1}=\frac{2g}{a}\int_{\DRR g{-1,2g-1}}\lambda_{g}\lambda_{g-1}-\frac{\left(2g-a\right)}{a}\int_{\DRR g{a-1,2g-a-1}}\lambda_{g}\lambda_{g-1}.
\end{equation}
(the $a$ parameter in this formula is the ``logarithmic order'' at the first marking). Since the twisted DR cycle is polynomial \cite{spelier2024polynomiality,PixtonZagier}, the above equality is an equality of polynomials in $a$. Therefore, we rewrite $I_{g}$ as
\begin{align}
I_{g} & =\left.\left[\frac{2g}{a}\left(\int_{\DRR g{-1,2g-1}}\lambda_{g}\lambda_{g-1}-\int_{\DRR g{a-1,2g-1-a}}\lambda_{g}\lambda_{g-1}\right)\right]\right|_{a=0}\nonumber \\
 & =-2g{\rm Coef}_{a}\int_{\DRR g{a-1,2g-1-a}}\lambda_{g}\lambda_{g-1}.
\end{align}

\medskip

\subsection{Step 3: Pixton formula for the twisted DR cycle}

To compute the integral $I_{g}$, we apply the definition of the twisted DR cycle as a Pixton class or equivalently Hain's formula~\cite{hain2011normal}. Indeed, we use once again that $\lambda_{g}\lambda_{g-1}$ vanishes on the complement of the moduli space of rational tails curves. Except the trivial graph, there is only one rational tail contribution (one vertex of genus $g$ and one vertex genus $0$ with two points) but its linear contribution in $a$ vanishes. Therefore the only contribution to $I_{g}$ comes from the trivial graph and we get:
\begin{equation}
I_{g}=\frac{-2g}{2^{g}}{\rm Coef}_{a}\int_{\overline{\mathcal{M}}_{g,2}}\lambda_{g}\lambda_{g-1}\exp\left(-\kappa_{1}+a^{2}\psi_{0}+\left(2g-a\right)^{2}\psi_{1}\right)
\end{equation}
 Extracting the linear coefficient in $a$, we get
\begin{equation}
I_{g}=\frac{1}{2^{g-1}}\int_{\overline{\mathcal{M}}_{g,2}}\lambda_{g}\lambda_{g-1}\left(2g\right)^{2}\psi_{1}\exp\left(-\kappa_{1}+\left(2g\right)^{2}\psi_{1}\right).\label{eq: Ig kappa psi}
\end{equation}

\medskip

\subsection{Step 4: Faber intersection number conjecture}

First, by standard properties of the $\kappa$-classes (see for instance \cite[Lemma 2.3]{pixton2013tautological}), we obtain: 
\begin{align}
\exp\left(-\kappa_{1}t\right) & =1+\sum_{n\geq1}\left(\sum_{m\geq1}\sum_{\substack{i_{1}+\cdots+i_{m}=n\\
i_{1},\dots,i_{m}>0
}
}\frac{\left(-1\right)^{m}}{m!}\frac{\kappa_{i_{1},i_{2},\dots,i_{m}}}{i_{1}!\cdots i_{m}!}\right)t^{n},
\end{align}
where $\kappa_{i_{1},i_{2},\dots,i_{m}}=\pi_{*}\left(\psi_{n}^{i_{1}+1}\cdots\psi_{n+m-1}^{i_{m}+1}\right)$ and $\pi:\overline{\mathcal{M}}_{g,n+m}\rightarrow\overline{\mathcal{M}}_{g,n}$. Using this identity and expanding $\exp\left(\left(2g\right)^{2}\psi_{1}\right)$ in Eq.~(\ref{eq: Ig kappa psi}), we get
\begin{align}
I_{g} & =\frac{1}{2^{g-1}}\sum_{d=0}^{g-2}\sum_{m\geq1}\sum_{i_{1}+\cdots+i_{m}=g-1-d}\frac{1}{\prod_{j=1}^{m}i_{j}!}\frac{\left(-1\right)^{m}}{m!}\frac{\left(2g\right)^{2d+2}}{d!}\int_{\overline{\mathcal{M}}_{g,2}}\lambda_{g}\lambda_{g-1}\kappa_{i_{1},i_{2},\dots,i_{m}}\psi_{1}^{d+1}.\nonumber \\
 & +\frac{1}{2^{g-1}}\frac{\left(2g\right)^{2g}}{\left(g-1\right)!}\int_{\overline{\mathcal{M}}_{g,2}}\lambda_{g}\lambda_{g-1}\psi_{1}^{g}
\end{align}
We now compute the integral using the following lemma. 
\begin{lem}
[Faber intersection number formula]Suppose that $g\geq1$, $n\geq1$, $d\geq0$ and $i_{1},\dots,i_{m}\geq1$. We have
\begin{equation}
\int_{\overline{\mathcal{M}}_{g,2}}\lambda_{g}\lambda_{g-1}\psi_{0}^{0}\psi_{1}^{d+1}\kappa_{i_{1},\dots,i_{m}}=\frac{\left(2g-1+m\right)!\left(2g-1\right)!!}{\left(2g-1\right)!\left(2d+1\right)!!\prod_{i=1}^{m}\left(2i_{i}+1\right)!!}\cdot c_{g}
\end{equation}
where $c_{g}=\int_{\overline{\mathcal{M}}_{g,1}}\psi_{0}^{g-1}\lambda_{g}\lambda_{g-1}=\frac{1}{2^{2g-1}\left(2g-1\right)!!}\frac{\left|B_{2g}\right|}{2g}$. 
\end{lem}

\begin{proof}
Using the pull-back property of $\psi$-classes, we get 
\[
\int_{\overline{\mathcal{M}}_{g,2}}\lambda_{g}\lambda_{g-1}\psi_{0}^{0}\psi_{1}^{d+1}\kappa_{i_{1},\dots,i_{m}}=\int_{\overline{\mathcal{M}}_{g,2+m}}\lambda_{g}\lambda_{g-1}\psi_{0}^{0}\psi_{1}^{d+1}\prod_{j=1}^{m}\psi_{j+1}^{i_{j}+1}.
\]
We then conclude using Faber's intersection number formula. Although the standard formulation of Faber's formula requires all $\psi$-class exponents to be positive, it still applies in the presence of a single exponent equal to zero. This extension is, for instance, apparent in the the proof of the formula given in \cite{blot2025faber}.
\end{proof}
Therefore, we obtain
\begin{align}
I_{g} & =\left|B_{2g}\right|\vast(\frac{1}{2^{3g-2}}\frac{1}{\left(2g-1\right)!}\sum_{d\geq0}^{g-2}\frac{\left(2g\right)^{2d+1}}{d!\left(2d+1\right)!!}\sum_{m\geq1}\sum_{\substack{i_{1}+\cdots+i_{m}=g-1-d\\
i_{1},\dots,i_{m}>0
}
}\frac{\left(-1\right)^{m}}{m!}\frac{\left(2g-1+m\right)!}{\prod_{j=1}^{m}i_{j}!\left(2i_{j}+1\right)!!}\nonumber \\
 & \qquad\qquad+\frac{1}{2^{3g-2}}\frac{\left(2g\right)^{2g-1}}{\left(g-1\right)!\left(2g-1\right)!!}\vast)
\end{align}
We denote by $J_{g}$ the expression in parenthesis. In the last step we establish that $J_{g}=1$, thereby completing the proof. 

\medskip

\subsection{Step 5: A combinatorial identity}

We use $\left(2l+1\right)!=\left(2l+1\right)!!\cdot l!\cdot2^{l}$ and re-organize the formula as follows
\begin{align}
J_{g} & =g\cdot\sum_{d=0}^{g-2}\frac{\left(2g\right)^{2d}}{\left(2d+1\right)!2^{2d}}\sum_{m\geq1}\sum_{\substack{i_{1}+\cdots+i_{m}=g-1-d\\
i_{1},\dots,i_{m}>0
}
}{2g-1+m \choose m}\frac{1}{\prod_{j=1}^{m}\left(-1\right)\left(2i_{j}+1\right)!2^{2i_{j}}}\nonumber \\
 & +g\cdot\frac{\left(2g\right)^{2g-2}}{\left(2g-1\right)!2^{2g-2}}.
\end{align}
Introducing $S\left(z\right)=\frac{\sinh\left(z/2\right)}{z/2}=\sum_{i\geq0}\frac{z^{2k}}{2^{2k}\left(2k+1\right)!}$, we obtain
\begin{align}
J_{g} & =g\left[z^{2g-2}\right]\sum_{m\geq0}{2g-1+m \choose m}S\left(2gz\right)\left(1-S\left(z\right)\right)^{m}.
\end{align}
Using $\sum_{m\geq0}{k+m \choose m}t^{m}=\left(\frac{1}{1-t}\right)^{k+1}$, we obtain
\begin{equation}
J_{g}=g\left[z^{2g-2}\right]\frac{S\left(2gz\right)}{\left(S\left(z\right)\right)^{2g}}=\frac{1}{2^{2g}}\left[1/z\right]\frac{\sinh\left(gz\right)}{\sinh\left(z/2\right)^{2g}}.
\end{equation}
Using $\frac{\sinh\left(2gz/2\right)}{\sinh\left(z/2\right)^{2g}}=\sum_{k=0}^{g-1}{2g \choose 2k+1}\coth^{2k+1}\left(z/2\right)$, we obtain 
\begin{equation}
J_{g}=\frac{1}{2^{2g}}\left[1/z\right]\sum_{k=0}^{g-1}{2g \choose 2k+1}\coth^{2k+1}\left(z/2\right).
\end{equation}

\begin{lem}
We have

\begin{equation}
\left[\frac{1}{x}\right]\coth^{2k+1}\left(x\right)=1,\quad k\geq0.
\end{equation}
\end{lem}

\begin{proof}
We use
\begin{equation}
\coth\left(x\right)=\sum_{n=0}^{\infty}\frac{2^{2n}B_{2n}x^{2n-1}}{\left(2n\right)!}
\end{equation}
to obtain
\begin{equation}
\left[\frac{1}{x}\right]\coth^{2k+1}\left(x\right)=\frac{2^{2k}}{\left(2k\right)!}\sum_{n_{1}+\cdots+n_{2k+1}=k}{2k \choose 2n_{1},\dots,2n_{2k+1}}B_{2n_{1}}\cdots B_{2n_{2k+1}}.
\end{equation}
The statement follows from \cite[Theorem 2b]{Dilcher}. 
\end{proof}
Therefore, we conclude that
\begin{equation}
J_{g}=\frac{1}{2^{2g}}\sum_{k=0}^{g-1}{2g \choose 2k+1}2=\frac{1}{2^{2g}}\frac{\left(1+1\right)^{2g}}{2}2=1.
\end{equation}

\bibliographystyle{alpha}
\bibliography{MD-Hodge}

\newcommand{\etalchar}[1]{$^{#1}$}
\begin{thebibliography}{BHP{\etalchar{+}}23}

\bibitem[BHP{\etalchar{+}}23]{BHPSS2020pixton}
Younghan Bae, David Holmes, Rahul Pandharipande, Johannes Schmitt, and Rosa
  Schwarz.
\newblock Pixton's formula and {Abel}-{Jacobi} theory on the {Picard} stack.
\newblock {\em Acta Math.}, 230(2):205--319, 2023.

\bibitem[BLS24]{blot2024strong}
Xavier Blot, Danilo Lewanski, and Sergey Shadrin.
\newblock {On the strong DR/DZ conjecture}.
\newblock {\em arXiv preprint arXiv:2405.12334}, 2024.

\bibitem[BR16]{BR2016}
Alexandr Buryak and Paolo Rossi.
\newblock {Double ramification cycles and quantum integrable systems}.
\newblock {\em Letters in mathematical physics}, 106(3):289--317, 2016.

\bibitem[BR24]{blot2024meromorphic}
Xavier Blot and Paolo Rossi.
\newblock {Meromorphic differentials, twisted DR cycles and quantum integrable
  hierarchies}.
\newblock {\em arXiv preprint arXiv:2408.13806}, 2024.

\bibitem[BSS25a]{blot2024master}
Xavier Blot, Adrien Sauvaget, and Sergey Shadrin.
\newblock The master relation for polynomiality and equivalences of integrable
  systems.
\newblock {\em Bull. Lond. Math. Soc.}, 57(2):599--604, 2025.

\bibitem[BSS25b]{blot2025faber}
Xavier Blot, Sergey Shadrin, and Ishan~Jaztar Singh.
\newblock Faber's socle intersection numbers via gromov--witten theory of
  elliptic curve.
\newblock {\em Bulletin of the London Mathematical Society}, 2025.

\bibitem[Bur15a]{Buryak15}
Alexandr Buryak.
\newblock {Double ramification cycles and integrable hierarchies}.
\newblock {\em Communications in Mathematical Physics}, 336(3):1085--1107,
  2015.

\bibitem[Bur15b]{Buryak_Hodge15}
Alexandr Buryak.
\newblock {Dubrovin-{Z}hang hierarchy for the {H}odge integrals}.
\newblock {\em Commun. Number Theory Phys.}, 9(2):239--272, 2015.

\bibitem[CSS21]{CSS}
Matteo Costantini, Adrien Sauvaget, and Johannes Schmitt.
\newblock {Integrals of $\psi$-classes on twisted double ramification cycles
  and spaces of differentials}.
\newblock 12 2021.

\bibitem[Dil96]{Dilcher}
Karl Dilcher.
\newblock Sums of products of {Bernoulli} numbers.
\newblock {\em J. Number Theory}, 60(1):23--41, 1996.

\bibitem[DSvZ22]{delecroix2022admcycles}
Vincent Delecroix, Johannes Schmitt, and Jason van Zelm.
\newblock {admcycles-a Sage package for calculations in the tautological ring
  of the moduli space of stable curves}.
\newblock {\em Journal of Software for Algebra and Geometry}, 11(1):89--112,
  2022.

\bibitem[FP18]{farkas2018moduli}
Gavril Farkas and Rahul Pandharipande.
\newblock {The moduli space of twisted canonical divisors}.
\newblock {\em Journal of the Institute of Mathematics of Jussieu},
  17(3):615--672, 2018.

\bibitem[Hai11]{hain2011normal}
Richard Hain.
\newblock {Normal functions and the geometry of moduli spaces of curves}.
\newblock {\em arXiv preprint arXiv:1102.4031}, 2011.

\bibitem[Hol21]{Holmes}
David Holmes.
\newblock Extending the double ramification cycle by resolving the
  {Abel}-{Jacobi} map.
\newblock {\em J. Inst. Math. Jussieu}, 20(1):331--359, 2021.

\bibitem[JPPZ17]{JPPZ17}
Felix Janda, Rahul Pandharipande, Aaron Pixton, and Dimitri Zvonkine.
\newblock {Double ramification cycles on the moduli spaces of curves}.
\newblock {\em Publications math{\'e}matiques de l'IH{\'E}S}, 125(1):221--266,
  2017.

\bibitem[Pix]{PixtonZagier}
Aaron Pixton.
\newblock {DR cycle polynomiality and related results. Published on A. Pixton
  personal website at https://websites.umich.edu/ pixton/papers/DRpoly.pdf}.

\bibitem[Pix13]{pixton2013tautological}
Aaron Pixton.
\newblock {\em The tautological ring of the moduli space of curves}.
\newblock PhD thesis, Princeton University, 2013.

\bibitem[SAK79]{satsuma1979internal}
Junkichi Satsuma, Mark~J Ablowitz, and Yuji Kodama.
\newblock On an internal wave equation describing a stratified fluid with
  finite depth.
\newblock {\em Physics Letters A}, 73(4):283--286, 1979.

\bibitem[Spe24]{spelier2024polynomiality}
Pim Spelier.
\newblock {Polynomiality of the double ramification cycle}.
\newblock {\em arXiv preprint arXiv:2401.17421}, 2024.

\end{thebibliography}

\end{document}